\documentclass[12pt]{amsart}
\usepackage{amssymb}
\usepackage{amsfonts}
\usepackage{latexsym}
\usepackage{amscd}

\newtheorem{lemma}{Lemma}
\newtheorem{corollary}[lemma]{Corollary}
\newtheorem{theorem}[lemma]{Theorem}
\newtheorem{proposition}[lemma]{Proposition}
\newtheorem{remark}[lemma]{Remark}

\newtheorem{definition}[lemma]{Definition}
\newtheorem{example}[lemma]{Example}

\newcommand\limind{\mathop{\oalign{\hfill $\rm lim$\hfill\cr$\longrightarrow$\cr}}}

\usepackage[all]{xy}

\input xy
\xyoption{all}

\begin{document}

\title[Row-finite equivalents  exist only for  row-countable graphs]{Row-finite equivalents exist only for \\ row-countable graphs}
\author{Gene Abrams and Kulumani M. Rangaswamy}
\address{Department of Mathematics, University of Colorado,
Colorado Springs CO 80933 U.S.A.} \email{abrams@math.uccs.edu, \ krangasw@math.uccs.edu}

\begin{abstract}
If $E$ is a not-necessarily row-finite graph, such that each vertex of $E$ emits at most countably many edges, then a {\it desingularization} $F$ of $E$ can be constructed as described in \cite{AA3} or \cite{R}.  The desingularization process has been effectively used to establish various characteristics of the Leavitt path algebras of not-necessarily row-finite graphs.    Such a desingularization $F$ of $E$ has the properties that:  (1) $F$ is row-finite, and (2) the Leavitt path algebras $L(E)$ and $L(F)$ are Morita equivalent.   We show here that for an arbitrary graph $E$, a graph $F$ having properties (1) and (2) exists (we call such a graph $F$ a \emph{row-finite equivalent of} $E$) if and only if $E$ is row-countable; that is, $E$ contains no vertex $v$ for which $v$ emits uncountably many edges.

\end{abstract}



\maketitle

\begin{center}
 Dedicated to Ken Goodearl on the occasion of his 65th birthday.
 \end{center}

\bigskip

The notion of a Leavitt path algebra was originally defined and investigated for row-finite graphs (i.e., graphs for which each vertex emits at most finitely many edges); see e.g. \cite{AAP1} and \cite{AMP}.  Subsequently, the Leavitt path algebras of more general graphs were investigated in \cite{AA3}; more precisely, those graphs for which the vertices are allowed to emit an  infinite (but at most countably infinite) number of edges.  (We call such a graph  {\it row-countable}.)   One of the methods used in \cite{AA3} to establish various results in this more general situation is as follows: associate with  the given row-countable graph $E$ a row-finite graph $F$ (a so-called {\it desingularization} of $E$) for which the Leavitt path algebras $L(E)$ and $L(F)$ are  Morita equivalent, then apply known results about the Leavitt path algebras of row-finite graphs to conclude some structural property of  $L(F)$, then transfer this property back to $L(E)$ via the equivalence.

Subsequent to \cite{AA3},  the notion of a Leavitt path algebra has been investigated in settings where there are no restrictions placed on the cardinality of either the vertex set or edge set of the underlying graph $E$ (we refer to such as {\it unrestricted} graphs; these are called {\it uncountable}  graphs in \cite{G}).   Ken Goodearl's article  \cite{G} contains an overarching discussion  of the germane ideas which allow for the passing of information from countable  graphs to unrestricted graphs.
  See also, e.g.,  \cite{ABR}, \cite{AR1}, and \cite{Ibero} for additional  analyses of Leavitt path algebras of unrestricted  graphs.

Broadly speaking, here's a three step Procedure by which a number of properties of  Leavitt path algebras of unrestricted graphs have been established.

\medskip

Step 1: \  Establish the property for the Leavitt path algebras of row-finite graphs.

 \smallskip

Step 2: \  Use the aforementioned desingularization process  to realize up to Morita equivalence the Leavitt path algebra  of a row-countable graph as the Leavitt path algebra of a row-finite graph.  Then show that the property in question is preserved by Morita equivalence.

  \smallskip

 Step  3: \ Use \cite[Proposition 2.7]{G} to realize the Leavitt path algebra of an unrestricted  graph as the direct limit of subalgebras, each of which arises as  the Leavitt path algebra of an appropriate type of   row-countable subgraph.     (``Appropriate" here means a {\it CK-subgraph}, see \cite[2.3]{G}).  Then show that the property in question is preserved by direct limits.     (This step of the Procedure is described as part of the \textsc{Modus Operandi} in \cite[Section 3]{G}.)

\medskip

 Specific representative examples of how this three step Procedure has been played out in full can be found, for example,  in the verification that the monoid of finitely generated projective left modules over any Leavitt path algebra is unperforated (see \cite[Theorem 5.8]{G}), as well as in the verification of the {\it Cuntz Krieger Uniqueness Theorem} (see \cite[Theorem 3.6]{G}).

As an aside, we note that  many properties of row-finite graphs (Step 1 of the Procedure) are in fact themselves established by an analysis similar to Step 3 of the Procedure, as follows:  one first verifies a property for finite graphs, and then establishes the same property for row-finite graphs by realizing   a Leavitt path algebra of a row-finite graph as the direct limit of subalgebras, each of which is isomorphic to the Leavitt path algebra of a finite graph (using  \cite[Lemma 3.2]{AMP}), and showing that the property in question is preserved by direct limits.

The row-countable graphs play two roles in the Procedure:  they are ``small" enough so that they can be analyzed using results about row-finite graphs (Step 2), while they are simultaneously ``ubiquitous" enough to cover via direct limits any Leavitt path algebra (Step 3).    It is then natural to ask whether both of these   intermediary roles of the row-countable graphs are in fact necessary ingredients to complete the Procedure.  That is, might it be possible to collapse Steps 2 and 3 into a single step?   More formally, we  ask

\medskip

 Question 1:  Can we realize the Leavitt path algebra of an unrestricted graph as the direct limit of subalgebras, each of which  arises as the Leavitt path algebra of a row-{\it finite} CK-subgraph of the original graph?  (Rephrased:  Can we bypass Step 2?) \ \
 and

 \smallskip

 Question 2:   Is there a process by which we may  realize up to Morita equivalence the Leavitt path algebra  of an {\it unrestricted} graph as the Leavitt path algebra of a row-finite graph?  (Rephrased:  Can we bypass Step 3?)

 \medskip

 It is well-known that the answer to Question 1 is {\it no}.  For instance, the Leavitt path algebra of the graph having  one vertex and infinitely many loops at that vertex provides a counterexample.

The primary goal of this short note is to show that the answer to Question 2 is {\it no} as well.     More precisely,  in  Theorem \ref{MainTheorem}  we show that for a given graph $E$, there exists a row-finite graph $F$ for which  $L(E)$ is Morita equivalent to $L(F)$ if and only if $E$ is row-countable.


\medskip

We recall that a graph $E = (E^0, E^1, r,s)$ has vertex set $E^0$, edge set $E^1$, and source and range functions $s,r$ respectively.  We call a vertex $v\in E^0$ \emph{regular} in case $1 \leq |s^{-1}(v)| < \infty$; otherwise, $v$ is called \emph{singular}.   The singular vertices consist of the \emph{sinks} (i.e., vertices which emit no edges) and the \emph{infinite emitters} (i.e., vertices which emit infinitely many edges).  An infinite emitter is   \emph{countable} (resp., \emph{uncountable}) according to whether the set of edges $s^{-1}(v)$ is countably infinite (resp., uncountably infinite).  The graph $E$ is called {\it row-finite} (resp., {\it row-countable}) in case $E$ contains no infinite (resp., uncountable) emitters.   Additional germane definitions and various notation may be found in the  previously cited works.

Our focus in this note is on $L_K(E)$, the Leavitt path
algebra
of $E$. We define $L_K(E)$ and give a
few examples.

\begin{definition}\label{definition}  {\rm Let $E$ be any
directed graph, and $K$ any field.
The {\em Leavitt path algebra} $L_K(E)$ {\em of $E$ with
coefficients in $K$} is
the $K$-algebra generated by a set $\{v \mid v\in E^0\}$ of
pairwise orthogonal idempotents,
together with a set of variables $\{e,e^* \mid e \in E^1 \}$,
which satisfy the following
relations:

(1) $s(e)e=er(e)=e$ for all $e\in E^1$.

(2) $r(e)e^*=e^*s(e)=e^*$ for all $e\in E^1$.

(3) (The ``CK1 relations") \ $e^*e'=\delta _{e,e'}r(e)$ for all
$e,e'\in E^1$.

(4) (The ``CK2 relations") \ For every nonsingular vertex $v$ of $E$, $v=\sum _{\{ e\in E^1\mid s(e)=v
\}}ee^*$.
}
\end{definition}

We will sometimes denote $L_K(E)$ simply by $L(E)$ for
notational convenience.
The set
$\{e^*\mid e\in E^1\}$ will be denoted by $(E^1)^*$. We let
$r(e^*)$
denote $s(e)$, and we let $s(e^*)$ denote $r(e)$. If $\mu =
e_1
\dots e_n$ is a path, then we denote by $\mu^*$ the element
$e_n^*
\dots e_1^*$ of $L_K(E)$.  We  view  paths $\mu$ in $E$ as
elements of $L_K(E)$, and often refer such a path as a  {\it
real} path, to distinguish it from elements of the form
$\mu^*$ of $L_K(E)$, which we refer to as {\it ghost} paths.

Many well-known algebras arise as the Leavitt path algebra of
a graph.
For example, the classical Leavitt $K$-algebra $L_K(1,n)$ for
$n\ge 2$; the full $n\times n$ matrix algebra ${\rm M}_n(K)$
over $K$; and the Laurent polynomial
algebra $K[x,x^{-1}]$ arise, respectively, as the Leavitt path
algebras of the ``rose with $n$ petals" graph $R_n$ ($n\geq
2$); the oriented line graph $A_n$ having $n$ vertices; and the
``one vertex, one loop" graph $R_1$ pictured here.

$$R_n \ =  
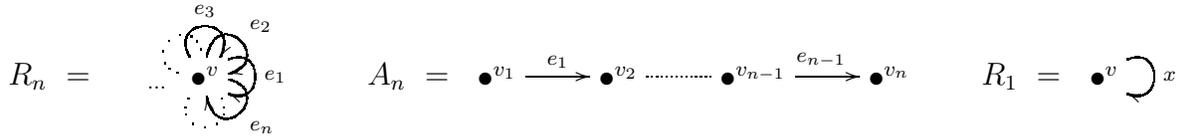
\xymatrix{ & {\bullet^v} \ar@(ur,dr) ^{e_1}
\ar@(u,r) ^{e_2}
\ar@(ul,ur) ^{e_3} \ar@{.} @(l,u) \ar@{.} @(dr,dl) \ar@(r,d)
^{e_n}
\ar@{}[l] ^{\ldots} } \ \  \ \ \ \  \ A_n \ = \
\xymatrix{{\bullet}^{v_1} \ar [r] ^{e_1} & {\bullet}^{v_2}
\ar@{.}[r] & {\bullet}^{v_{n-1}} \ar [r]
^{e_{n-1}} & {\bullet}^{v_n}} \ \ \ \ \  \ R_1 \ = \
\xymatrix{{\bullet}^{v} \ar@(ur,dr) ^x}$$

\begin{definition}
{\rm Let $E   = (E^0,E^1,r,s)$ be an unrestricted directed graph (i.e., there is no restriction placed on the cardinalities of the vertex set $E^0$ or the edge set $E^1$).   By a {\it row-finite equivalent} of $E$ we mean a directed graph $F$ for which:
 \begin{enumerate}
 \item  $F$ is row-finite, and
 \item  the Leavitt path algebras $L(E)$ and $L(F)$ are Morita equivalent.
  \end{enumerate}
  }
  \end{definition}


For any edge $e\in E^1$ it is always the case that $ee^*$ is an idempotent in $L(E)$, and that if $e\neq f \in E^1$ then $ee^*$ and $ff^*$ are orthogonal.   The following is thereby straightforward.

 \begin{lemma}\label{orthogidemp}
    Suppose $w$ is an uncountable emitter in $E$.
   Let the edges being emitted at $w$ be denoted
  by $\{e_{\alpha}|\alpha \in A\}$.
 Then $\{e_\alpha e_\alpha^* | \alpha \in A\}$ is an uncountable set of pairwise orthogonal idempotents in $wL(E)w$.
   Rephrased, the set $\{e_\alpha e_\alpha^* | \alpha \in
 A\}$ is an uncountable set of pairwise orthogonal idempotents in ${\rm End}_{L(E)}(L(E)w)$.
\end{lemma}

We now proceed to show that if $F$ is a row-finite graph (we emphasize that $F$ {\it is}
 allowed to have uncountably many vertices and / or edges), then there is no finitely generated projective
left $L(F)$-module whose endomorphism ring contains an uncountable set of pairwise orthogonal idempotents.


For any ring $R$ we denote by $\mathcal{V}(R)$ the semigroup of isomorphism classes of finitely generated projective left $R$-modules, with operation $\oplus$.  For any graph $F$, the semigroup $M_F$ is defined as the abelian semigroup generated by $\{a_v | v \in F^0\}$, with relations given by
 $$a_v = \sum_{e \in s^{-1}(v)}a_{r(e)}$$
 for each nonsingular vertex $v$ of $F$.

\begin{proposition}\label{Threeresultsprop}
\begin{enumerate}
\item Let $F$ be row-finite (but possibly with uncountably many vertices and~/~or  edges). Then there is an isomorphism of semigroups $\varphi: M_F \rightarrow \mathcal{V}(L(F))$.

\item Let $E$ and $F$ be unrestricted graphs.  If $\Phi: L(E)Mod \rightarrow L(F)Mod$ is a Morita equivalence,
and $P\in \mathcal{V}(L(E))$, then $\Phi(P)\in \mathcal{V}(L(F))$.

\item  Let $F$ be  row-finite.  If $\Phi: L(E)Mod \rightarrow L(F)Mod$ is a Morita equivalence, then for each $w\in E^0$ there is
  an isomorphism of left $L(F)$-modules  $\Phi(L(E)w)\cong \oplus_{i=1}^n L(F)v_i$ for
some sequence $v_1,v_2,...,v_n$ of (not necessarily distinct) vertices of $F$.

\end{enumerate}
\end{proposition}

\begin{proof}   (1)   Since any row-finite graph is the direct limit of its finite CK-subgraphs (see e.g. \cite[Lemmas 3.1, 3.2]{AMP} or \cite[Proposition 2.6]{G}), the proof is identical to that given in \cite[Theorem 3.5]{AMP}.

(2)    This is established in  \cite [Corollary 5.6]{G}.

 (3) By (2),  $\Phi(L(E)w)$ is in $\mathcal{V}(L(F))$.
But by (1),  each object in $\mathcal{V}(L(F))$ is isomorphic to an
 $L(F)$-module of the indicated type.
\end{proof}


 We now establish some properties of $L(F)$ for row-finite graphs $F$.

\begin{proposition}\label{countablymanypqstar}
  Let $F$ be row-finite, and let $v,v'\in F^0$.   Then there are at most countably
 many distinct expressions of the form $pq^*$ in $L(F)$ for which $s(p)=v, r(q^*)=v',$ and $r(p)=r(q)$.
\end{proposition}

\begin{proof}
  Because $F$ is row-finite, for any positive integer $N$ and any vertex $v$ there exists at most finitely many
 distinct paths of length $N$ which emanate
from $v$. So there are at most countably many distinct paths in $F$ which emanate from $v$.  Similarly there are at most countably
 many distinct (real) paths which emanate from $v'$, so that there are at most countably many ghost paths of the form $q^*$ having $r(q^*)=s(q)=v'$.
Now any nonzero expression of the form $pq^*$ corresponds to a pair of directed  paths $p$ and $q$
 for which $s(p)=v, r(q^*)=v'$, and $r(p)=r(q)$, and the result follows.
\end{proof}

\begin{corollary}\label{countabledimensionvL(F)v}
 Let $F$ be row-finite, and let   $v,v'\in F^0$.
Then ${\rm dim}_K(vL_K(F)v')$ is at most countable.
\end{corollary}

\begin{proof}
  As a $K$-space, $L_K(F)$ is spanned by expressions of the form
$$\{pq^* \ | \ p,q \mbox{ are paths in } F \mbox{ with } r(p)=r(q)\}.$$
(This set is typically not linearly independent, but that is not of concern here.)  Then
 $vL_K(F)v'$ is spanned by expressions of the form
 $$\{pq^* \ | \ p,q \mbox{ are paths in } F \mbox{ with } s(p)=v, s(q)=r(q^*)=v' \mbox{ and }r(p)=r(q)\}.$$  The result
  now follows from Proposition \ref{countablymanypqstar}.
\end{proof}

  \begin{corollary}\label{rowfinitegivescountabledim}

  Let $F$ be row-finite, and let $v_1,v_2,...,v_n$ be a sequence of (not necessarily distinct) vertices of $F$.   Then the $K$-dimension
 of the $K$-algebra ${\rm End}_{L(F)}(\oplus_{i=1}^n L(F)v_i)$ is at most countable.
\end{corollary}

\begin{proof}   Since each $v_i$ is idempotent in $L(F)$, it is standard that as a ring we have
$${\rm End}_{L(F)}(\oplus_{i=1}^n L(F)v_i)\cong R,$$
 where $R$ is the $n\times n$
 matrix ring having $(i,j)^{{\rm th}}$ entry $R_{i,j}=v_iL(F)v_j$ for each pair $1\leq i,j \leq n$.   This isomorphism is clearly seen to be a $K$-algebra map as well.   Since
  ${\rm dim}_K(R_{i,j})$ is at most countable for each pair $i,j$ by Corollary \ref{countabledimensionvL(F)v}, the result follows.
\end{proof}

\begin{lemma}\label{dimensionatleastcard(S)}
  If $B$ is any $K$-algebra, and $B$ contains a set $S$ of nonzero orthogonal idempotents, then ${\rm dim}_K(B) \geq {\rm card}(S)$.
\end{lemma}

  \begin{proof}
  Suppose $\sum_{i=1}^n k_ie_i = 0$ with $k_i\in K$ and $e_i\in S$.   Then by hypothesis each $e_i\neq 0$, and $e_ie_j = \delta_{ij}e_i$ for all
 $i,j$.   So multiplying the given equation on the right by $e_i$ gives $k_ie_i = 0$, whence $k_i = 0$ and we are done.
\end{proof}

Putting all the pieces of the puzzle together, we now have the tools to conclude

   \begin{proposition}\label{onedirection}
   Suppose $E$ is not row-countable.   Then $E$ admits no row-finite equivalent.
\end{proposition}

\begin{proof}
Let $F$ be a row-finite graph.  By Corollary \ref{rowfinitegivescountabledim}, for any sequence $v_1,v_2,...,v_n$ of vertices of $F$, ${\rm End}_{L(F)}(\oplus_{i=1}^n L(F)v_i)$ has at most countable $K$-dimension.  So, by Lemma  \ref{dimensionatleastcard(S)}, ${\rm End}_{L(F)}(\oplus_{i=1}^n L(F)v_i)$ cannot contain an uncountable set of nonzero orthogonal idempotents.

 Now arguing to the contrary, suppose $\Phi: L(E)Mod \rightarrow L(F)Mod$ is a Morita equivalence. Let $w$ denote an uncountable emitter in $E$.
  Then, by Proposition \ref{Threeresultsprop}(3),
$\Phi(L(E)w) \cong \oplus_{i=1}^n L(F)v_i$
 for some vertices $v_1,v_2,...,v_n$ of $F$.   As Morita equivalence preserves endomorphism rings,  this would yield
 ${\rm End}_{L(E)}(L(E)w)\cong {\rm End}_{L(F)}(\oplus_{i=1}^n L(F)v_i)$.
   But as noted in Lemma \ref{orthogidemp}, ${\rm End}_{L(E)}(L(E)w)$ contains an uncountable set of orthogonal idempotents,   while  ${\rm End}_{L(F)}(\oplus_{i=1}^n L(F)v_i)$ does not.
  \end{proof}

Proposition \ref{onedirection} establishes one direction of our main result.  We now review the appropriate constructions which allow us to build row-finite equivalents.   The germane ideas  appear in \cite{DT} and \cite{AA3}.

 \begin{definition}
 {\rm If $v_0$ is a countable  emitter in $E$, then by \emph{adding a tail at $v_0$} we mean a process by which we modify the graph $E$, as follows.   We first order the edges $e_1, e_2, e_3, \ldots$ of $s^{-1}(v_0)$. Then we add new vertices $v_1,v_2, ...$ and new edges $f_1,f_2,...$  to $E$ at $v_0$
as pictured here:
$$\xymatrix{ {\bullet}^{v_0} \ar[r]^{f_1} & {\bullet}^{v_1} \ar[r]^{f_2} & {\bullet}^{v_2} \ar[r]^{f_3} &
{\bullet}^{v_3} \ar@{.>}[r] & }$$
Next, we remove the original set of edges  $e_1,e_2,e_3,...$ from the graph.  Finally,  for each  removed edge $e_j$, we add a new
edge $g_j$ having  $s(g_j) = v_{j-1}$ and $r(g_j) = r(e_j)$.
\hfill $\Box$}
\end{definition}

We note that the countability of $s^{-1}(v_0)$ allows for the construction of a sequence of edges and vertices (as displayed above) for which, given any two vertices $v_i$ and $v_j$ with $i\leq j$, there is a unique path $p_{i,j}$  having $s(p_{i,j})=v_i$ and $r(p_{i,j}) = v_j$.   Such a configuration would not be possible if $s^{-1}(v_0)$ were uncountable.  This distinction will manifest later in our main result.

We also note (for later use) that in a tail added at a countable emitter $v_0$, the CK2 relation in the new graph yields $v_{m-1} = f_mf_m^* + g_mg_m^*$, so that $f_mf_m^* = v_{m-1} - g_mg_m^*$ for each $m\geq 1$.

\begin{example}\label{desingexample}
{\rm Let  $R_\infty$ denote the {\it infinite rose} graph having one vertex $v_0$  and countably many loops $\{e_i\}$ at $v_0$.  Then adding a tail at $v_0$ yields the new graph
$$
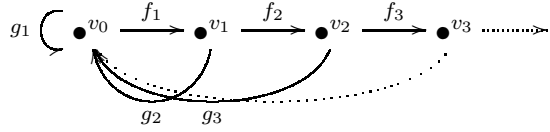
\xymatrix{
 {\bullet}^{v_0} \ar[r]^{f_1} \ar@(ul,dl)_{g_1}[] &
  {\bullet}^{v_1} \ar[r]^{f_2} \ar@(d,d)[l]^{g_2}
   & {\bullet}^{v_2} \ar[r]^{f_3} \ar@(d,d)[ll]^{g_3}
    & {\bullet}^{v_3} \ar@{.>}[r] \ar@(d,d)@{.>}[lll] &
    &
}$$
}
\end{example}

\begin{example}\label{infiniteedgesexample}
{\rm Let $E_\infty$ denote the {\rm infinite edges graph}
$$\xymatrix{{\bullet}^v  \ar[r]^{(\infty)} & {\bullet}^w}$$ where the label $(\infty)$ denotes the infinite set of edges
$E^1=\{e_i\mid i\ge 1\}$ with $s(e_i)=v$ and $r(e_i)=w$.
 Then adding a tail at $v$ yields the new graph
$$
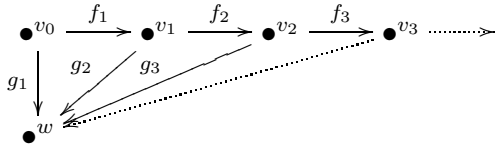
\xymatrix{ {\bullet}^{v_0} \ar[r]^{f_1} \ar[d]_{g_1} & {\bullet}^{v_1} \ar[r]^{f_2} \ar[dl]_{{g_2}} & {\bullet}^{v_2} \ar[r]^{f_3} \ar[dll]_{g_3} &
               {\bullet}^{v_3} \ar@{.>}[r] \ar@{.>}[dlll] &  \\
             {\bullet}^{w}  &  &  &  & }$$
}
\end{example}

\begin{remark}
 {\rm  In general, as noted in \cite{DT},   different orderings of the edges of a countable emitter $s^{-1}(v_0)$ in a graph $E$ may give rise to nonisomorphic graphs via the process of adding a tail at $v_0$.
}
\end{remark}

We are now in position to establish the main result of this note.

  \begin{theorem}\label{MainTheorem}  Let $E$ be an unrestricted graph. Then  $E$ admits a row-finite equivalent if and only if  $E$ is row-countable.
  \end{theorem}

  \begin{proof}
 If $E$ is not row-countable, then $E$ admits no row-finite equivalent by Proposition \ref{onedirection}.

So suppose that $E$ is row-countable; we produce a row-finite equivalent for $E$.
  To do so we use the process  described in \cite[Theorem 5.2]{AA3} as a guide.
Specifically, let $F$ be a row-finite graph constructed from $E$ by adding a tail at each infinite   emitter of $E$, as described above.   (We use here the hypothesis that each infinite emitter is in fact a countable emitter.)   By identifying each infinite emitter $v$ in $E$ with the corresponding vertex $v_0$ of $F$, we may view $E^0$ as a subset of $F^0$.

 By \cite[Proposition 5.1]{AA3},
there exists a monomorphism of algebras $\phi:L(E)\hookrightarrow L(F)$, defined as follows.  If $v\in E^0$ we have two cases. If $v$ is not an infinite emitter, then $v\in F$ as well, and we define
$\phi(v)=v$. If $v$ is an infinite emitter,  then $v$ has been
replaced in $F$ by an infinite tail beginning with $v_0$, so we define in this case $\phi(v)=v_0$.
Now consider $e\in E^1$. If $s(e)$ is not an infinite emitter then we set $\phi(e)=e$, and $\phi(e^*)=e^*$. In contrast, if
$s(e)$ is an infinite emitter, then when adding a tail at $s(e)$ in the construction of $F$ we would have named $e$ as $e_i$ for
some $i\ge 1$.   In this situation, we define  $\phi(e_i)=f_1\dots f_{i-1}g_i$, and $\phi(e^*_i)=g^*_i f^*_{i-1}\dots f^*_1$.
We extend $\phi$ linearly and multiplicatively to all of $L(E)$ to achieve the desired homomorphism.   That $\phi$ is a monomorphism is established in \cite[Proposition 5.1]{AA3}.

Recall that $L(E)$ has the collection of sums of distinct vertices as a set of local units.  In other words, if we label the vertices $E^0=\{v_\alpha | \alpha \in A\}$,
 then the set of idempotents
 $$T = \{\sum_{j\in A_i}v_j \ | \  A_i \mbox{ is a finite subset of }A\}$$
  is a set of local units for $L(E)$.  Since $E^0 \subseteq F^0$, we may view the elements of $T$ as elements of $L(F)$ as well.

 We pick an arbitrary element $t \in T$, and establish that $tL(E)t\cong tL(F)t$.
   Suppose $t =\sum_{j\in A_x}v_j$ for the finite subset $A_x$ of $A$.
    We consider the restriction $\phi|_{tL(E)t}:tL(E)t \hookrightarrow L(F)$. Since $\phi(t)=t$, we have that $\phi|_{tL(E)t}$ is indeed a monomorphism from $tL(E)t$ to $
tL(F)t$, so that we only need to see that this restriction is onto.

The corner algebra $tL(F)t$ is the linear span of the monomials of the form $pq^*$ where $r(p)=r(q)$ and both $p$ and $q$ are paths in $F$ that begin at any vertex $v_l$
with $l \in A_x$.  (In particular, $v_l \in E^0$.)  Note that any path $p$ having this property must be of the form $p_1\dots p_r f_1\dots f_{j-1}$ where $p_n$ are either edges already in $E$ or
new paths in $F$ of the form $f_1\dots f_{h-1}g_h$, and $f_m$ are edges along a tail. Any such $p_n$ is  in the image of
$\phi$ by definition. So it is enough to show that $(f_1\dots f_{j-1})((f')^*_{j'-1}\dots (f')^*_1)$ is in the image of $\phi$, since these are the only expressions in $L(F)$ which start and end at a vertex of $E$.

But this is done exactly as in the proof of \cite[Propostion 5.2]{AA3}; we give the essential details of the argument here for completeness.     (Alternately, we could also use  the construction given in \cite[Lemma 6.7]{T2} to achieve the same result.)  First note that for this element to be nonzero it must be the case that $j=j'$ and $f_m=f'_m$ for every $m\leq j$.  Now we have:
\begin{eqnarray}
 (f_1\dots f_{j-1})(f^*_{j-1}\dots f^*_1) &=& (f_1\dots f_{j-2})(v_{j-2}-g_{j-1}g^*_{j-1})(f^*_{j-2}\dots f^*_1)     \nonumber \\
   &=& (f_1\dots f_{j-2})(f^*_{j-2}\dots f^*_1)-(f_1\dots f_{j-2}g_{j-1})(g^*_{j-1}f^*_{j-2}\dots f^*_1) \nonumber \\
   & = & \nonumber \hdots
\end{eqnarray}
If we continue this process of replacing $f_mf_m^*$ by $v_{m-1} - g_mg_m^*$, we reach an expression of the form
$$v_0-g_1g^*_1-\sum_{i=2}^{j-1} (f_1\dots f_{i-1}g_i)(g^*_i f^*_{i-1}\dots f^*_1),$$
which we see is precisely
$$\phi(v
-e_1e^*_1-\sum_{i=2}^{j-1} e_i e^*_i).$$
This shows that $\phi|_{tL(E)t}:tL(E)t \to tL(F)t$ is surjective, and thus an
isomorphism of $K$-algebras. Moreover, these isomorphisms are defined in such a way that the following diagram
commutes whenever $t\leq t'$ in the standard partial order on $T$ (i.e., whenever $tt' = t't = t$ in $T$).
$$\xymatrix{ tL(E)t \ar[rr]^{\phi|_{tL(E)t}} \ar@{^{(}->}[d]^i  &   &  tL(F)t  \ar@{^{(}->}[d]^i \\
             t'L(E)t' \ar[rr]^{\phi|_{t'L(E)t'}} &   & t'L(F)t' }$$
In particular, we then get that the two direct limit rings
$$\mathop{\oalign{\hfill $\limind $\hfill\cr$ \scriptstyle{t \in T} $\cr}} tL(E)t \mbox{\hskip 0.5cm and
\hskip 0.5cm } \mathop{\oalign{\hfill $\limind $\hfill\cr$ \scriptstyle{t \in T} $\cr}} tL(F)t $$ are
isomorphic. But the first of these rings is just $L(E)$, since  $T$ is a set of local
units for $L(E)$.  Thus we have shown that
$$\mathop{\oalign{\hfill $\limind $\hfill\cr$ \scriptstyle{t \in T} $\cr}} tL(F)t \cong L(E).$$
Now suppose $w_0$ is an infinite emitter in $E$. Let $w_i$ be any vertex in $F$ which arises in the tail added at $w_0$, and let $p_i$ denote the path $p_i = f_1f_2\cdots f_i$ in $F$ having $r(p_i) = w_i$.  Define $\rho_i:
L(F)w_i \rightarrow L(F)w_0$ by $x \mapsto xp_i^*$, and define $\pi_i: L(F)w_0 \rightarrow L(F)w_i$ by $y \mapsto
yp_i$.  Then $\rho_i$ and $\pi_i$ are left $L(F)$-module homomorphisms, and, since $p_i^* p_i = w_i$, we conclude that
$L(F)w_i$ is isomorphic to a direct summand of $L(F)w_0$ as left $L(F)$-modules.

Since $L(F)\cong \bigoplus_{v\in F^0}L(F)v$ as left $L(F)$-modules, and $L(F)$ is a generator for $L(F)-Mod$, the
previous paragraph demonstrates that the $L(F)$-module $\bigoplus_{v\in E^0}L(F)v \cong \mathop{\oalign{\hfill $\limind
$\hfill\cr}}_{t \in T} L(F)t$ is in fact a generator for $L(F)-Mod$.

We now apply \cite[Theorem 2.5]{AM} to conclude that the rings $\mathop{\oalign{\hfill $\limind $\hfill\cr}}_{t\in
T} {\rm End}_{L(F)}(L(F)t)$ and $L(F)$ are Morita equivalent. But ${\rm End}_{L(F)}(L(F)t)\cong tL(F)t$, so that by
the previously displayed isomorphism we have that $L(F)$ and $L(E)$ are Morita equivalent, and we are done.
\end{proof}

\begin{remark}
{\rm
 Theorem \ref{MainTheorem} yields that if  $E$ contains an uncountable emitter, then  $E$ admits no row-{\it finite} equivalent.  In fact, more can be said:  if $E$ contains an uncountable emitter, then $E$ admits no row-{\it countable} equivalent, since if $L(E)$ is Morita equivalent to $L(F)$ for the row-countable graph $F$, then using the row-finite equivalent $G$ for $F$ guaranteed by Theorem \ref{MainTheorem} we would have $L(E)$ Morita equivalent to $L(G)$, which cannot happen by Proposition \ref{onedirection}.
}
\end{remark}

An historical comment is in order here.   For the graph C$^*$-algebraists, the non-existence of sinks in a graph has played an important role.    Thus the analog of the aforementioned  ``trading-in" process in the context of C$^*$-algebras seeks to  trade in an unrestricted graph for a graph that is not only row-finite, but contains no sinks as well; in other words, a graph which contains no singular vertices.   Rephrased, the C$^*$-algebraists are interested in a \emph{desingularized equivalent} of a graph, which for notational convenience is simply referred to as a \emph{desingularization} of a graph.   With this as context, we make the following definition.

\begin{definition}
 {\rm If $v_0$ is a sink in a graph $E$, then by \emph{adding a tail at $v_0$} we mean attaching a graph of
the form
$$\xymatrix{ {\bullet}^{v_0} \ar[r] & {\bullet}^{v_1} \ar[r] & {\bullet}^{v_2} \ar[r] & {\bullet}^{v_3} \ar@{.>}[r] & }$$
to $E$ at $v_0$.
}
\end{definition}

By using exactly the same ideas as those presented in the proof of Theorem \ref{MainTheorem}, we see that  if we start with a row-countable graph $E$, and build a graph $F$ by adding a tail at every infinite emitter {\it and} every sink, then $F$ contains no singular vertices, and   $L(E)$ is Morita equivalent to $L(F)$.    Thus we may somewhat strengthen the statement of  Theorem \ref{MainTheorem}, as follows.

\begin{theorem}\label{threeconditionstheorem}
Let $E$ be an unrestricted graph. The following are equivalent:
   \begin{enumerate}
   \item  $E$ admits a row-finite equivalent.
    \item $E$ admits a desingularization; that is, there exists a row-finite graph $F$ having no sinks for which $L(E)$ and $L(F)$ are Morita equivalent.
     \item $E$ is row-countable.
     \end{enumerate}
\end{theorem}

\begin{remark}
{\rm
Historically, Step 2 of the aforementioned three-step Procedure has been completed by using a desingularization of $E$.  In fact, Step 2 may be completed by using any  row-finite equivalent of $E$, desingularized or not.
}
\end{remark}


\begin{thebibliography}{99}


\bibitem{AAP1}  \textsc{G. Abrams, G. Aranda Pino},  The Leavitt path algebra of a graph, \emph{J.  Algebra} \textbf{293}(2) (2005), 319-334.


\bibitem{AA3}\textsc{G. Abrams, G. Aranda Pino}, Leavitt path algebras of arbitrary graphs,
\emph{Houston J. Math}, \textbf{34}(2) (2008), 423-442.

\bibitem{ABR}  \textsc{G. Abrams, J. Bell, and K.M. Rangaswamy}, On prime, non-primitive von Neumann regular rings, in preparation.



\bibitem{AR1} \textsc{G. Abrams and K.M. Rangaswamy}, Regularity conditions for arbitrary Leavitt path algebras, \emph{Alg. Rep. Thy.} \textbf{13}  (2010), 319-334.

\bibitem{AM}  \textsc{P.N. \'Anh, L. M\'arki}, Morita equivalence
for rings without identity, \emph{Tsukuba J. Math} \textbf{11}(1) (1987), 1-16.


\bibitem{AMP} \textsc{P. Ara, M.A. Moreno, E. Pardo}, Nonstable K-Theory for graph algebras,
\emph{Alg. Rep. Thy.} \textbf{10} (2007), 157-178.


\bibitem{Ibero} \textsc{G. Aranda Pino, D. Mart\'{\i}n Barquero, C. Mart\'{\i}n Gonz\'{a}lez, and M. Siles Molina}  Socle theory for Leavitt path algebras of arbitrary graphs, \emph{Rev. Mat. Iberoamericana} \textbf{26}(2) (2010), 611-638.


\bibitem{DT} \textsc{D. Drinen, M. Tomforde}, The C*-algebras of arbitrary graphs, \emph{Rocky Mountain J. Math}
\textbf{35}(1) (2005), 105--135.

\bibitem{G}  \textsc{K. Goodearl} Leavitt path algebras and direct limits, in ``Rings, Modules and Representations",  Contemporary Mathematics series (2009),  165--188.

\bibitem{R}  \textsc{I. Raeburn}, Graph algebras. CBMS Regional Conference Series in Mathematics 103, Published for the Conference Board of the Mathematical
              Sciences, Washington, DC, 2005.
      ISBN 0-8218-3660-9.

\bibitem{T2} \textsc{M. Tomforde}, Uniqueness theorems and ideal structure for Leavitt path algebras, \emph{J.
Alg.} \textbf{318} (2007), 270-299.

\end{thebibliography}
\end{document}